\documentclass[10pt,twoside,reqno]{amsart}
\usepackage[all]{xy}
\usepackage{amssymb,latexsym,amsthm,amsmath,mathtools,dsfont,mathrsfs}
\usepackage{enumitem,color}

\usepackage[top=1in,bottom=1in,left=1.2in,right=1.2in]{geometry}
\usepackage{makecell}
\usepackage{mathtools}

\usepackage{longtable}
\usepackage{framed}
\usepackage{float}
\usepackage{hyperref}
\usepackage{url}
\usepackage[capitalize,nameinlink]{cleveref} %theorem and lemmas will not mess up
\usepackage{comment} % to comment big chunk
\usepackage{xcolor} % to get rid of boxes around hyperlinks
\usepackage{enumitem}
\long\def\symbolfootnote[#1]#2{\begingroup%
\def\thefootnote{\fnsymbol{footnote}}\footnote[#1]{#2}\endgroup}

\newcommand{\overbar}[1]{\mkern 1.5mu\overline{\mkern-1.5mu#1\mkern-5.0mu}\mkern 3.5mu}

\makeatletter
\def\imod#1{\allowbreak\mkern10mu({\operator@font mod}\,\,#1)}
\makeatother
\makeatletter
\renewcommand*\env@matrix[1][*\c@MaxMatrixCols c]{%
  \hskip -\arraycolsep
  \let\@ifnextchar\new@ifnextchar
  \array{#1}}
\makeatother

\hypersetup{
    colorlinks,
    linkcolor={red!80!black},
    citecolor={red!80!black},
    urlcolor={blue!80!black}
}
\newtheorem{theorem}{Theorem}[section]
\newtheorem{lemma}[theorem]{Lemma}
\newtheorem{corollary}[theorem]{Corollary}

\newtheorem*{theorem*}{Theorem}
\theoremstyle{definition}

\newtheorem{example}[theorem]{Example}
\newtheorem{conjecture}[theorem]{Conjecture}

\numberwithin{equation}{section}
\newcommand{\ignore}[1]{}

\newcommand{\mynote}[1]{}
\title[On codes in (generalized) symmetric groups]{On codes in (generalized) symmetric groups}
\author[Subrata Barman]{Subrata Barman}
\email{mp21010@iisermohali.ac.in}
\address{Indian Institute of Science Education and Research Mohali, Sector 81, Mohali 140306, India}

\author[Pushpendra Singh]{Pushpendra Singh}
\email{pushpendra@iisermohali.ac.in}
\address{Indian Institute of Science Education and Research Mohali, Sector 81, Mohali 140306, India}
\thanks{}
\date{\today}
\subjclass[2020]{20B30, 20C30, 20F55, 05C25, 94B60}
\keywords{Symmetric Groups, Codes, Conjugacy Classes, Characters, Cayley Graphs}

\begin{document}

\begin{abstract}
In this article, we show that, for the symmetric group $S_n$, the Young subgroup $S_{\lambda}$, with $l(\lambda)\geq 3$, is not a code with respect to a conjugacy class of $S_n$. This provides a partial answer to \cite[Problem 4.1]{FL26}. We then provide a characterization of conjugacy classes $X_i$, such that subgroups $S_{(n-2,1,1)}$, $S_{(n-3,2,1)}$, and $Y_k$ for $k\leq 3$ are codes for $X_1 \cup X_2$. Furthermore, we describe some codes in the finite Coxeter groups of type $B_n$, $C_n$, $D_n$, and in the generalized symmetric group $C_m \wr S_n$.
\end{abstract}

\maketitle

\section{Introduction}
Let $G$ be a finite group and $A$ be a subset of $G$ such that $1 \not \in A$ and $A^{-1}=A$. Then the \textit{Cayley graph} $\rm{Cay}(G,A)$ is a simple graph with vertex set $G$ and and edge set $E=\{(g,h)\;|\;gh^{-1}\in S\}$. A perfect $r$-code \cite{Kra86}, $B$ is a subset of $G$ such that every vertex of $\rm{Cay}(G, A)$ is at most $r$-distance away from exactly one element of $B$. If $r=1$, then $B$ is called a \textit{perfect code}. We say $B$ is a \textit{total perfect code} if every element of $G$ is exactly one distance away from a unique element of $B$. On the other hand, let $A$ be a subset of $G$, we say $A$ divides $rG$, if there exists a subset $B$ of $G$ such that there are exactly $r$ pairs $(a,b) \in A \times B $ such that $ab=g$ for all $g \in G$. Then we can write $A\cdot B=r G$. We note that $A\cdot B=r G$ if and only if each vertex of $\rm{Cay}(G, A)$ has exactly $r$ neighbours in $B$. Such a subset $B$ is called a \textit{code} with respect to $A$ in $G$. 

The authors in \cite{GL20} described various families of such codes for symmetric and special linear groups. For symmetric groups $S_n$, the Young subgroup $Y_{k}=S_k \times S_{n-k},k>3$, is proven to be a code in \cite{FL26} with respect to conjugacy class $x^G$  where $x$ has exactly one cycle of length $2^i$ for $0\leq i\leq j$ and $2^j\leq k \leq 2^{j+1}$ and all other cycles have length at lest $k+1$. Furthermore, in \cite{Feng_21}, for ${\rm PGL}(2,q)$, the authors proved that the dihedral group $D_{q+1}$ is a code for the subset $A=\{g \in {\rm PGL}(2,q)\;|\; g^{q+1}=1, q^2 \neq 1\}$.

In this article, we show that, for the symmetric group $S_n$, the Young subgroup $S_{\lambda}$, with $l(\lambda)\geq 3$, is not a code with respect to a conjugacy class of $S_n$. We provide a characterization of conjugacy classes $X_i$, such that subgroups $S_{(n-2,1,1)}$, $S_{(n-3,2,1)}$, and $Y_k$ for $k\leq 3$ are group codes for $X_1 \cup X_2$. Furthermore, we describe some codes in finite Coxeter groups of type $B_n$, $C_n$, $D_n$, and in the generalized symmetric group $C_m \wr S_n$. 

The article is organized as follows. In section \ref{2}, we give preliminaries on the representation theory of $S_n$ and its connection with group codes. In sections \ref{3}, \ref{Sec:4}, \ref{Sec:5}, we give our main results on the finite Coxeter groups of type $B_n, C_n$, $D_n$ and in the generalized symmetric group $C_m \wr S_n$, respectively.

\section{Preliminaries} \label{2}
In this section, we outline the connection between the group codes and the representation theory of the finite group $G$. This relation is described in \cite{FL26}, and we refer to it for more details.

Let $G$ be a finite group and $\mathbb{C}[G]$ be the group algebra over the complex field. Let $\chi$ be an irreducible character of $G$ and $I^{\chi}=(\mathbb{C}[G]) c^{\chi}$ be the associated left ideal where $c^{\chi}=\frac{\chi(1)}{|G|} \sum\limits_{g}\chi(g^{-1})g$. For a subset $A \subseteq G$, we denote $\overline{A}=\sum\limits_{a\in A}a \in \mathbb{C}[G]$. Let $\rm {Irr}(G)$ denote the set of all irreducible characters of $G$.

\begin{lemma}\cite[Lemma 2.2]{FL26}\label{l1}
With the above notations, we have the following
\begin{enumerate}[label=\normalfont(\roman*)]
\item
$rG=A\cdot B$ if and only if $r\overline{G}=\overline{A}\cdot \overline{B}$
\item
Let $X$ be a conjugacy class of $x\in G$ and $\chi \in \rm{Irr}(G)$ then $\overline{X}w=w\overline{X}=\frac{|X|\chi(x)}{\chi(1_G)}w$, for all $w\in I_{\chi}$.
\end{enumerate}
\end{lemma}

We note that part ${\rm (ii)}$ of the above lemma can be generalized in the following way.
\begin{lemma}\label{l2}
Let $\chi \in \rm{Irr}(G)$ and $X=\cup X_i$ where $X_i$ are distinct conjugacy classes of $ x_i \in G$. Then
$\overline{X}w=w\overline{X}=\frac{w}{\chi(1_G)}{\sum\limits_{i}|X_i|\chi(x_i)}$ for all $w \in I_{\chi}$.
\end{lemma}

\subsection{Representation theory of $S_n$} The irreducible representations of $S_n$ are in bijective correspondence with partitions of $n$. A partition of $n$ is a tuple of non negative integers $\lambda=(\lambda_1,\lambda_2,\dots,\lambda_l)$ where $\lambda_1\geq \lambda_2\dots\geq \lambda_l$ and $\lambda_1+\lambda_2 +\dots+\lambda_l=n$. We can define a partial order on the set of all partitions of $n$. For $\lambda=(\lambda_1,\lambda_2,\dots,\lambda_l)$ and $\mu=(\mu_1,\mu_2,\dots,\mu_k)$, We say $\lambda$ dominates $\mu$ if $\lambda_1 + \lambda_2 + \dots  + \lambda_i\geq \mu_1+ \mu_2 + \dots +\mu_i$  for all $i\geq 1$ and denote by $\lambda\trianglerighteq \mu$.

A Young diagram of $\lambda$ consists of $n$ blocks arranged in $l$ left-justified rows, where the $i$th row has $\lambda_i$ blocks. A $\lambda$-tableau $t$ is a Young diagram with its blocks filled with $1,2,\dots,n$. We note that $\sigma \in S_n$ acts on $t$ in a natural way by permuting entries of $t$. The column stabilizer of $t$ is the subgroup $C_t$ of $S_n$ that preserves the columns of $t$. That is, $\sigma\in C_t$ if and only if $\sigma(i)$ is in the same column as $i$ for each $i \in \{1,...,n\}$.

We define an equivalence relation on the set of $\lambda$-tableau with $t_1 \sim t_2$ if they have the same entries in each row. We denote the set of equivalence classes $[t]$ with $T^{\lambda}$ and note that $|T^{\lambda}|=\frac{n!}{\lambda_1!\lambda_2!\dots \lambda_l!}$. Further, $S_n$ acts on $T^{\lambda}$ by defining $\sigma[t] =[\sigma.t]$. Let $M^{\lambda}=\mathbb CT^{\lambda}$ and $S^\lambda$ be the subspace of $M^\lambda$ spanned by the $\lambda$-polytabloids $e_t$, where $e^t=\sum\limits_{\sigma \in C_t}\rm{sgn}(\sigma)\sigma[t]$. The subspace $S^\lambda$ is called Specht module and the representation $\phi^{\lambda}: S_n \to GL(S^{\lambda})$ is an irreducible representation of $S_n$. We refer to \cite{steinberg2012representation} for more details. 

For $\lambda=(\lambda_1,\lambda_2,\dots,\lambda_l)$, we denote $S_{\lambda}$ as subgroup $S_{\lambda_1}\times S_{\lambda_2}\times\dots\times S_{\lambda_l}$. If $\lambda=(n-k,k)$ then we use $Y_k$ to denote the subgroup $S_{k}\times S_{n-k}$.

\begin{lemma}\label{l3}
The Specht module $S^{\mu}$ is an irreducible component of $(\mathbb{C}S_n)\overline{S_{\lambda}}$ if and only if $\mu \trianglerighteq \lambda$. Moreover, $(\mathbb{C}S_n)\overline{S_{\lambda}}= \bigoplus\limits_{\theta \trianglerighteq \lambda} U^{\theta}$, where $U^{\theta} \subseteq I^{\theta}$ is a direct sum of $K_{\theta{\lambda}}$ minimal left ideals of $\mathbb{C}S_n$.
\end{lemma}

\begin{proof}
The proof follows from \cite{FL26}.
\end{proof}

We note that the \cite[Lemma 3.2]{FL26} can be generalized for all young subgroups $S_\lambda$ with $\lambda \neq (n)$ and for the union of conjugacy classes. We give the proof below.

\begin{lemma}\label{l4}
Let $r$ be a positive integer and $X=\cup X_i$ where $X_i$ are conjugacy classes of $S_n$.  Then $rS_n=XS_{\lambda}$ if and only if $\sum\limits_{i}|X_i|\chi_{\mu}(x_i)=0$, with $x_i \in X_i$ for all $\mu \trianglerighteq \lambda$ and $\mu \neq (n)$.
\end{lemma}

\begin{proof}

First suppose that $rS_n=X\cdot S_{\lambda}$, which implies $\overline{X}~\overline{S_\lambda}\in I^{(n)}$. Let $\mu \trianglerighteq \lambda$ and $\mu \neq (n)$, then we have $c^{\mu}\overline{S_\lambda} \neq 0$ and if $\sum\limits_{i}|X_i|\chi_{\mu}(x_i)\neq 0$ then $c^{\mu}\overline{X}=\alpha_{\mu}c^{\mu}$ for some $\alpha_\mu \in \mathbb C \backslash \{0\}$ by Lemma \ref{l2}. Then we have
\[\alpha_\mu(c^\mu\overline{S_\lambda})=c^\mu\overline{X}~\overline{S_\lambda}=c^\mu(r\overline{S_n})=0\]
which is a contradiction since $\alpha_\mu(c^\mu\overline{S_\lambda})\neq 0$.

Conversely assume that $\sum\limits_{i}|X_i|\chi_{\mu}(x_i)=0$ for all $\mu \trianglerighteq \lambda$ and $\mu \neq (n)$. Then by Lemma \ref{l1}{\rm (i)}, it is enough to show that $\overline{X}\cdot\overline{S_\lambda}\in I^{(n)}$ where $I^{(n)}$ is the $\mathbb
C$-vector space spanned by $\overline{S_n}$. If $\mu \not \trianglerighteq \lambda$ then by \cite[Lemma 2.1]{FL26}, we have $c^{\mu}\overline{S_\lambda}=0$ since $\overline{S_\lambda} \in \bigoplus\limits_{\theta\trianglerighteq \lambda}I^{\theta}$ by Lemma \ref{l3}. If $\mu \trianglerighteq \lambda$ and $\mu \neq (n)$ then $c_{\mu}\overline{X}=0$ by Lemma \ref{l2} and our assumption. Therefore, for any $\mu \neq (n)$, we have $c^{\mu}\overline{X}\overline{S_{\lambda}}=0$, which implies $\overline{X}\overline{S_{\lambda}} \in I^{(n)}$.

\end{proof}

\subsection{Character theory of symmetric groups}
Let $\lambda \vdash k$ and $\Lambda=(n-k,k)\vdash n$. Let $x\in S_n$ have the partition type $(1^{a_1},2^{a_2},3^{a_3},\dots)$. Let $X^{\Lambda}$ denote the irreducible character of $S_n$ corresponding to partition $\Lambda\vdash n$. Then we have
\[X^{\Lambda}(x)= \sum\limits_{\rho,\sigma} (-1)^{l(\sigma)}z_{\sigma}^{-1}\chi_{\rho\cup\sigma}^{\lambda}\binom{a}{\rho} \] 
summed over partitions $\rho,\sigma$ such that $|\rho|+|\sigma|=|\lambda|$. Here $l(\sigma)$ denotes the length of partition $\sigma$, $\rho\cup\sigma$ is a partition whose parts are those of $\rho$ and $\sigma$, $\chi^{\lambda}$ denotes the  irreducible character of $S_k$ corresponding to partition $\lambda\vdash k$, and $\chi_{\rho\cup\sigma}^{\lambda}$ is its value on conjugacy class of partition $\rho\cup\sigma \vdash k$, $z_\sigma=\prod\limits_{i\geq 1}i^{m_i}{m_i}!$ where $\sigma=(1^{m_1},2^{m_2},\dots)$ and $\binom{a}{\rho}=\prod\limits_{r\geq 1}\binom{a_r}{n_r}$ where $\rho=(1^{n_1},2^{n_2},\dots)$. For more details, we refer the reader to \cite{macdonald1998symmetric}. 

\begin{lemma}\cite{baker2016character} \label{characters}
Let $\mu \vdash n$ and $\chi_\mu$ be the corresponding irreducible character of $S_n$. Let $a_k$ denote the number of $k$-cycles in the cycle type of $x \in S_n$. Then we have the following:
\begin{enumerate}[label=\normalfont(\roman*)]
    \item $\chi_{(n)}(x)=1$
    \item $\chi_{(n-1,1)}(x)=a_1-1$
    \item $\chi_{(n-2,2)}(x)=a_2-a_1+\frac{a_1(a_1-1)}{2}$
    \item $\chi_{(n-3,3)}(x)=a_3+a_1a_2-a_2-\frac{a_1(a_1-1)}{2}+\frac{a_1(a_1-1)(a_1-2)}{6}$
    \item $\chi_{(n-2,1,1)}(x)=-a_2-a_1+1+\frac{a_1(a_1-1)}{2}$
    \item $ \chi_{(n-3,2,1)}(x)= a_1 -a_3-a_1(a_1-1)+\frac{a_1(a_1-1)(a_1-2)}{3}$
    \item $\chi_{(n-3,1,1,1)}(x) =\frac{(a_1-1)(a_1-2)(a_1-3)}{6}-(a_1-1)a_2+a_3 $

\end{enumerate}
\end{lemma}

\section{Codes in symmetric groups} \label{3}
In this section, we describe the existence and non-existence of codes for the symmetric group $S_n$ with respect to some conjugacy-closed subsets.

\begin{lemma}\label{ml1}
Let $S_n$ be the symmetric group of degree $n$. Let $\lambda=(n-2,2) \vdash n$ and $\mu=(n-2,1,1) \vdash n$ be partitions of $n$. Then for any $g \in S_n$, $\chi_{\lambda}(g)$ and $\chi_{\mu}(g)$ can not be zero simultaneously. 
\end{lemma}
\begin{proof}
Let $g \in S_n$ and $a_k$ denote the number of $k$-cycles in the cycle type of $g$. Then from Lemma \ref{characters}, we get that
$\chi_{\lambda}(g)=a_2-a_1+\frac{a_1(a_1-1)}{2}$ and  $\chi_{\mu}(g)=-a_2-a_1+1+\frac{a_1(a_1-1)}{2}$.\\
Suppose $\chi_{\lambda}(g)=\chi_{\mu}(g)=0$. Then upon solving $\chi_{\lambda}(g)-\chi_{\mu}(g)=0$, we get $2a_2=1$ which is a contradiction since $a_k$ are non negative integers.
\end{proof}

\begin{lemma}\label{ml2}
Let $S_n$ be the symmetric group of degree $n$. Let $\lambda= (\lambda_1,\lambda_2,\lambda_3) \vdash n$ be a partition of $n$ and let $S_\lambda$ be the  Young subgroup corresponding to $\lambda$. Then there does not exist any conjugacy class $X$ of $S_n $ such that $rS_n=XS_{\lambda}$ where $r \in \mathbb{N}$.

\end{lemma}

\begin{proof}
 Suppose there exist a conjugacy class $X$ such that $rS_n= XS_\lambda$, then by Lemma \ref{l1}, we get $\chi_\mu(g)=0$, where $g\in X$ and $\mu\trianglerighteq \lambda$, $\mu \neq (n)$. We observe that if $\mu = (\mu_1,\mu_2,\mu_3) \trianglerighteq \lambda$ then $(n-2,1,1) \trianglerighteq \mu$, further, we have $(n-2,2) \trianglerighteq (n-2,1,1)$. This would imply $\chi_{(n-2,2)}(g)=\chi_{(n-2,1,1)}(g)=0$ which cannot happen due to Lemma \ref{ml1}. This proves the result.
\end{proof}

\begin{theorem}
Let $S_n$ be the symmetric group of degree $n$. Let $\lambda= (\lambda_1,\lambda_2,\dots,\lambda_l) \vdash n$ with $l\geq 3$ and let $S_\lambda$ be the  Young subgroup corresponding to $\lambda$. Then there does not exist any conjugacy class $X$ of $S_n $ such that $rS_n=XS_{\lambda}$ where $r \in \mathbb{N}$.
\end{theorem}

\begin{proof}
 Let $\lambda= (\lambda_1,\lambda_2,\dots,\lambda_l) \vdash n$ with $l\geq 3$ then we note that both $(n-2,1,1)\trianglerighteq \lambda$ and $(n-2,2) \trianglerighteq \lambda$. Hence, the result holds using Lemma \ref{ml2}.
\end{proof}

\begin{theorem}\label{t2}
Let $k \leq 3$ and $n>2k$. Suppose $X=X_1 \cup X_2$ is the union of two distinct conjugacy classes of $S_n$ with $x_i\in X_i$. Let $a_l$ (resp. $b_l$) to denote number of $l$-cycles in $x_i$(resp. $x_j$).
\begin{enumerate}[label=\normalfont(\roman*)]
    \item
    If $k=1$ then $rS_n=XY_{k}$ if and only if $x_i,x_j$ have exactly one fixed point or  $x_i$ has no fixed points, $x_j$ has $b_1\geq 2$ fixed points with $|X_i|=(b_1-1)|X_j|$ for $i\neq j$.
    \item
    If $k=2$ then $rS_n=XY_{k}$ if and only if $x_i,x_j$ have exactly one fixed point, exactly one $2$- cycle or $x_i,x_j$ have exactly one fixed point, $x_i$ has no $2$-cycle, $x_j$ has $b_2$ $2$-cycle with $b_2 \geq 2$ and $|X_i|=(b_2-1)|X_j|$ or $x_i$ has no fixed points, $x_j$ has three fixed points with $|X_i|=2|X_j|$ and $a_2=b_2=0$. or $x_i$ has no fixed points, $x_j$ has two fixed points with $|X_i|=|X_j|$ and $a_2+b_2=1$ 
    \item
    If $k=3$, then $rS_n=XY_{k}$ if and only if $x_i,x_j$ have exactly one fixed point,exactly one $2$-cycle and no $3-$cycles or $x_i,x_j$ has exactly one fixed point, $x_i$ has no $2$-cycle, $x_j$ has $b_2$ $2$-cycle with $b_2 \geq 2$ with $|X_i|=(b_2-1)|X_j|$ and $x_i,x_j$ have no $3$-cycles or $x_i$ has no fixed points, no $2$-cycles, no $3$-cycles, $x_j$ has two fixed points, one $2$-cycle, no $3$-cycle and $|X_i|=|X_j|$ or $x_i$ has no fixed points, one $2$-cycle, $x_j$ has two fixed points, no $2$-cycles with $|X_i|=|X_j|$ and $a_3+b_3=2$ or $x_i$ has no fixed points, no $2$-cycles, $x_j$ has three fixed points, no $2$-cycles with $|X_i|=2|X_j|$ and $2a_3+b_3=2$.
\end{enumerate}
\end{theorem}

\begin{proof}
Let $x_i \in X_i$ and $x_j \in X_j$. We use $a_l$ (resp. $b_l$) to denote number of $l$-cycles in $x_i$(resp. $x_j$). For the reverse direction, using Lemma \ref{l4}, it is enough to show that  
\begin{equation}\label{eq:1}
|X_1|\chi_{\mu}(x_i) +|X_2|\chi_{\mu}(x_j)=0
\end{equation}
for all $\mu \trianglerighteq \lambda$ and $\mu \neq (n)$. 

\text{Case {\rm (i)}}. For $\lambda=(n-1,1)$, we have $\mu=(n-1,1)$ and so $\chi_{\mu}(x)=a_1-1$. First, we are given $a_1=b_1=1$. Thus, we have $\chi_{\mu}(x_i)=\chi_{\mu}(x_j)=0$, hence equation \ref{eq:1} is satisfied. 

For the next case, we have $\chi_{\mu}(x_i)=-1$ and $\chi_{\mu}(x_j)=b_1-1$ and $|X_i|=(b_1-1)|X_j|$. Thus, $-1\cdot(b_1-1)|X_j|+(b_1-1)|X_j|=0$. 

For forward direction, if $rS_n=XY_k$ then from Lemma \ref{l4} we get $(a_1-1)|X_i|+(b_1-1)|X_j|=0$, and so $a_1|X_i|+b_1|X_j|=|X_i|+|X_j|$. Since $a_1,b_1 \geq 0$ and $|X_i|,|X_j|>0$, then we will get the required two cases. 

\text{Case {\rm (ii)}}.
For $\lambda=(n-2,2)$, we have $\mu\in \{(n-1,1),(n-2,2)\}$. For $\mu=(n-1,1)$, equation \ref{eq:1} is satisfied for all cases similarly to case ${\rm (i)}$. Now we show the same for $\chi_\mu$ for $\mu=(n-2,2)$. We have $\chi_{\mu}(x)=a_2-a_1+\frac{a_1(a_1-1)}{2}$. For the first case, we have $a_1=b_1=1, a_2=b_2=1$. So $\chi_{\mu}(x_i)=\chi_{\mu}(x_j)=0$. Thus, equation \ref{eq:1} is satisfied. 
 
Next, we have $a_1=b_1=1$, $a_2=0,b_2\geq2$ and $|X_i|=(b_2-1)|X_j|$. We have $\chi_{\mu}(x_i)=-1, \chi_{\mu}(x_j)=b_2-1$. Thus, equation \ref{eq:1} is satisfied.
 
For the next case, we have $a_1=0, b_1=3$ with $|X_i|=2|X_j|$ and $a_2=b_2=0$. We have 
 \begin{align*}
\chi_{\mu}(x_i)|X_i|+\chi_{\mu}(x_j)|X_j|=0\\
\end{align*}

For the next case, we have $a_1=0,b_1=2$ with $|X_i|=|X_j|$ and $a_2+b_2=1$. We have 
 \begin{align*}
\chi_{\mu}(x_i)|X_i|+\chi_{\mu}(x_j)|X_j|&=a_2|X_i|+(b_2-2+\frac{2(2-1)}{2})|X_j|\\
&=(a_2+b_2-1)|X_j|=0
\end{align*}

For forward direction, suppose $\chi_{\mu}(x_i)|X_i| + \chi_{\mu}(x_j)|X_j|=0$
for $\mu=(n-2,2),(n-1,1)$. The computation for $\mu=(n-1,1)$ gives us $a_1=b_1=1$ or $a_1=0,b_1\geq 2$ and $|X_i|=(b_1-1)|X_j|$. For $\mu=(n-2,2)$, we get
$$\left(a_2-a_1+\frac{a_1(a_1-1)}{2}\right)|X_i|+\left(b_2-b_1+\frac{b_1(b_1-1)}{2}\right)|X_j|=0$$

If $a_1=b_1=1$, then above equation gives $(a_2-1)|X_i|+(b_2-1)|X_j|=0$ which implies that $a_2|X_i|+ b_2|X_j|= |X_i|+ |X|_j$. Since $a_2,b_2 \geq 0$ and $|X_i|,|X_j|>0$, then we get the first two cases. 

If $a_1=0,b_1\geq 2$ and $|X_i|=(b_1-1)|X_j|$. Then we have
$$a_2(b_1-1)|X_j|+\left(b_2-b_1+\frac{b_1(b_1-1)}{2}\right)|X_j|=0$$
Since $|X_j|\neq 0$. So, upon simplification, we get $b_1^2-3b_1-2a_2+2b_2+2a_2b_1=0$. Since all variables are non-negative, so $b_1\geq 4$ is not possible. For $b_1=3$, we get $a_2=b_2=0$ and $|X_i|=2|X_j|$. For $b_1=2$, we get $|X_i|=|X_j|$ and $a_2+b_2=1$ i.e $(a_2,b_2)=(1,0)$ or $(0,1)$.

\text{Case {\rm (iii)}}.
For $\lambda=(n-3,3)$, we have $\mu\in \{(n-1,1),(n-2,2),(n-3,3)\}$. For $\mu\in \{(n-1,1),(n-2,2)\}$, equation \ref{eq:1} is satisfied for all cases similarly to case ${\rm (i)}$ and ${\rm (ii)}$. Now we show the same for $\mu=(n-3,3)$.
For the first case, we have $a_1=b_1=1$, $a_2=b_2=1$, $a_3=b_3=0$. So $\chi_{\mu}(x_i)=\chi_{\mu}(x_j)=0$. Thus, equation \ref{eq:1} is satisfied.

Next, we have, $a_1=b_1=1$, $a_2=0,b_2\geq 2$, $a_3=b_3=0$. Then we have $\chi(x_i)=\chi(x_j)=0$ and so equation \ref{eq:1} is satisfied. Similarly direct substitution of next three cases gives $|X_i|\chi_{\mu}(x_i)+|X_j|\chi_{\mu}(x_j)=0$.

For forward direction, suppose $\chi_{\mu}(x_i)|X_i| + \chi_{\mu}(x_j)|X_j|=0$
for $\mu=(n-1,1),(n-2,2),(n-3,3)$. The Lemma \ref{l4}, for $\mu=(n-1,1)$ gives $a_1=b_1=1$ or $a_1=0,b_1\geq 2,|X_i|=(b_1-1)|X_j|$. Using $a_1=b_1=1$, the computation for $\mu=(n-2,2)$ gives $|X_i|(a_2-1)+|X_j|(b_2-1)=0$ and so $a_2=b_2=1$ or $a_2=0,b_2\geq 2,|X_i|=(b_2-1)|X_j|$. In both cases Lemma \ref{l4} for $\mu=(n-3,3)$ gives $a_3=b_3=0$ giving us first two cases. Putting $a_1=0,b_1\geq 2,|X_i|=(b_1-1)|X_j|$ for $\mu=(n-2,2)$ gives $b_1^2-3b_1-2a_2+2b_2+2a_2b_1=0$.
Putting $a_1=0,b_1\geq 2, |X_i|=(b_1-1)|X_j|,b_1^2-3b_1-2a_2+2b_2+2a_2b_1=0$ for $\mu=(n-3,3)$ gives $(b_1-1)(6a_3-4a_2+4b_2-2b_1-2a_2b_1)+6b_3=0$. These two equations upon solving give remaining three cases.

\end{proof}

\begin{example}
We provide an example of conjugacy classes satisfying properties in Theorem \ref{t2}. Let $X_1, X_2$ be conjugacy classes corresponding to partitions $\lambda_1,\lambda_2$ respectively. For case {\rm (i)}, in $S_9$, we can take $\lambda_1=(1,8),\lambda_2=(1,2,6)$ and $\lambda_1=(4,5),\lambda_2=(1,1,2,5)$ for respective subcases. For case {(ii)}, in $S_9$, we can take $\lambda_1=(1,2,6),\lambda_2=(1,2,3,3)$ and $\lambda_1=(1,4,4),\lambda_2=(1,2,2,4)$ and $\lambda_1=(2,7),\lambda_2=(1,1,7)$ for respective subcases. 

For case {(iii)}, the examples are, in $S_{11}$, $\lambda_1=(1,2,8),\lambda_2=(1,2,4,4)$ and in $S_9$, $\lambda_1=(1,4,4),\lambda_2=(1,2,2,4)$ and $\lambda_1=(3,6),\lambda_2=(1,1,1,6)$ for respective subcases.
\end{example}

\begin{theorem}\label{t3}
Let $\lambda=(n-2,1,1) \vdash n$ and $X=X_1\cup X_2$ be the union of two distinct conjugacy classes of $S_n$. Then $rS_n=XS_{\lambda}$ if and only if $x_i$ has no fixed points and $x_j$ has $ 2$ fixed points with $|X_i|=|X_j|$ and $a_2+b_2=1$.
\end{theorem}

\begin{proof}
We have $\lambda=(n-2,1,1)$. If $\mu \trianglerighteq \lambda$ and $\mu \neq (n)$ then $\mu \in \{(n-1,1),(n-2,2),(n-2,1,1)\}$. Therefore $x_i,x_j$ satisfy conditions of Theorem \ref{t2}{\rm(ii)}. Now we identify further conditions using $\mu=(n-2,1,1)$ and Lemma \ref{l4}. We observe that first three cases of $x_i,x_j$ of Theorem \ref{t2}{\rm(ii)} are not possible since from Lemma \ref{characters}, we get $\chi_{\mu}(x_i)=\chi_{\mu}(x_j)=-1$ and $\chi_{\mu}(x_i)=0,\chi_{\mu}(x_j)=-b_2 \neq 0$. From the fourth case, we have $a_1=0, b_1=2$,$|X_i|=|X_j|$ with $a_2+b_2=1$
Therefore, we get $\chi_{\mu}(x_i)=-a_2+1$ and $\chi_{\mu}(x_j)=-b_2$. 
Now
\begin{align*}
\chi_{\mu}(x_i)|X_i|+\chi_{\mu}(x_j)|X_j|=&(-a_2+1)|X_j|+(-b_2)|X_j|\\
=&|X_j|(1-a_2-b_2)=0
\end{align*}
The converse follows similarly to Theorem \ref{t2}{\rm(i)},{\rm(ii)}.
\end{proof}

\begin{theorem} \label{t4}
Let $\lambda=(n-3,2,1) \vdash n$ and $X=X_1\cup X_2$ be the union of two distinct conjugacy classes of $S_n$. Then $rS_n=XS_{\lambda}$ if and only if $x_i$ has no fixed points, no $2$-cycles, no $3$-cycles, and $x_j$ has two fixed points, one $2$-cycle, no $3$-cycle and $|X_i|=|X_j|$.
\end{theorem}

\begin{proof}
Let $\mu\trianglerighteq \lambda$ and $\mu \neq (n)$. Then $\mu \in \{(n-1,1),(n-2,2),(n-3,3),(n-2,1,1),(n-3,2,1)\}$. Using Theorem \ref{t3}, we have $a_1=0,b_1=2,|X_i|=|X_j|$, with $a_2+b_2=1$. Now using $\mu=(n-3,3)$, we get $|X_i|(a_3-a_2)+|X_j|(b_3+b_2-1)=|X_j|(a_3+b_3-2a_2)=0$.  

Now for $\mu=(n-3,2,1)$, we have $\chi_{\mu}(x_i)=-a_3$, $\chi_{\mu}(x_j)=-b_3$. So
\begin{align*}
\chi_{\mu}(x_i)|X_i|+\chi_{\mu}(x_j)|X_j|=&-a_3|X_i|-b_3|X_j|=-(a_3+b_3)|X_j|\\
\end{align*}
Since $|X_j|\neq 0$, so $a_3=b_3=0$ and so $a_2=0$ and $b_2=1$. Converse follows similarly as well.

\end{proof}

\begin{example}
The conjugacy classes $X_1,X_2$ corresponding to partitions $(4,5),(1,1,2,5)$ respectively in group $S_9$, satisfy the properties of both Theorem \ref{t3} and Theorem \ref{t4}.
\end{example}

\begin{theorem}\label{u4}
Let $S_n$ be the symmetric group of degree $n$. Let $\lambda= (\lambda_1,\lambda_2,\dots\lambda_l)$ with $l\geq 4$ be a partition of $n$ and let $S_\lambda$ be the  Young subgroup corresponding to $\lambda$. Then there does not exist distinct conjugacy classes $X_i,X_j$ of $S_n $ such that $rS_n=(X_i \cup X_j)S_{\lambda}$ where $r \in \mathbb{N}$.

\end{theorem}

\begin{proof}
Since $(n-3,2,1)\trianglerighteq \lambda$ so using Theorem \ref{t4},  we get $a_1=a_2=a_3=0$ and $b_1=2,b_2=1,b_3=0$. Now $(n-3,1,1,1)\trianglerighteq\lambda$ and we have $\chi_{(n-3,1,1,1)}(x_i)=\chi_{(n-3,1,1,1)}(x_j)=-1$. Thus, Lemma \ref{l4} is not satisfied. Hence, we have the result.
\end{proof}

\begin{conjecture}
Let $n\geq k\geq 5$ and $\lambda$ be a partition of $n$ with $l(\lambda)\geq k$. Suppose $X=\sqcup_{i=1}^{s} X_i$ with $1\leq s\leq k-2$ is the union of pairwise distinct conjugacy classes. Then $rS_n\neq XS_{\lambda}$ for any $r\in \mathbb Z_{\geq 0}$.
\end{conjecture}

\section{Codes in Coxeter groups of type \texorpdfstring{$B_n$}{Bn}, \texorpdfstring{$C_n$}{Cn} and \texorpdfstring{$D_n$}{Dn}}\label{Sec:4}

The finite Coxeter groups of type $B_n$ and $C_n$ are isomorphic. The group $B_n$ is the wreath product $C_2 \wr S_n= C_2^n \rtimes S_n$. We take $C_2 = \{ \pm 1 \}$ and the action of $S_n$ on $C_2^n$ is defined as follows:
\[\sigma\cdot(a_1,a_2,\dots,a_n)=(a_{\sigma(1)},a_{\sigma(2)},\dots,a_{\sigma(n)})\]

The group multiplication is defined as 
\[ [a_1,a_2,\dots,a_n;\sigma]\cdot [b_1,b_2,\dots,b_n;\tau]=[a_1b_{\sigma(1)},a_2b_{\sigma(2)},\dots,a_nb_{\sigma(n)};\sigma\tau]\]

We follow this convention from \cite{GAP4}, which considers the $S_n$ multiplication $\sigma\tau$ from left to right. Let $(\lambda_1,\lambda_2,\dots,\lambda_k)\vdash n$ and $(a_1,a_2,\dots,a_n) \in C_2^n$. We use symbol $\overline{\lambda}_i$ to correspond to the $\lambda_i$-cycle $(k_1,k_2,\dots,k_{\lambda_i})$ such that for the tuple $(a_{k_1},a_{k_2},\dots,a_{k_{\lambda_i}})$ we have $\prod a_{k_i}=-1$. The conjugacy classes of $B_n$ are in one-to-one correspondence with the signed partitions of $n$ denoted by $\left( \lambda_1^{s_1} \overbar{\lambda_1}^{t_1}\lambda_2^{s_2} \overbar{\lambda_2}^{t_2}\dots \lambda_k^{s_k} \overbar{\lambda_k}^{t_k} \right)$. Two elements of $B_n$ are conjugate if they have the same signed partitions. 

For example $[1,1,-1,-1;(1,2,3)], [1,-1,1,1;(1,2)(3,4)] \in C_2\wr S_3$ correspond to signed partitions $(\bar 1^1 \bar{3}^1)$ and $(2^1 \bar{2}^1)$ respectively.

The finite Coxeter group of type $D_n$ is an index two subgroup of $B_n$. As a set, it contains elements $[a_1,a_2,\dots,a_n;\sigma] \in B_n$ satisfying $\prod a_i=1$. We refer to \cite{Musili11} for more details. We define subgroups $H_0,H_1$ of $B_n$ as sets $\{[1,1,\dots,1;\sigma],[-1,-1,\dots,-1;\rho] |\;\sigma,\rho \in S_n\}$ and $\{[1,1,\dots,1;\sigma]|\; \sigma \in S_n\}$ respectively. We observe that $H_0 \cong C_2 \times S_n$ and $H_1 \cong S_n$.

\begin{lemma}\label{l:4.1}
Let $G$ be a finite group and $H, K$ be conjugate subgroups of $G$. Then for $r \in \mathbb N$, and a conjugacy closed subset $X$ of $G$, we have $rG=HX$ if and only if $rG=KX$.
\end{lemma}
\begin{proof}
The proof is immediate using $K=gHg^{-1}$, and the conjugation closed property of $X$.
%We have $K=gHg^{-1}$ for some $g \in G$. If $rG=HX$ then $rG=grGg^{-1}=gHXg^{-1}=gHg^{-1}gXg^{-1}$=$KX$.
\end{proof}

\begin{theorem} \label{coxeter}
Let $G=C_2 \wr S_n$ and $H$ be a subgroup of $G$ conjugate $H_0$. Let $X_1,X_2$ be the conjugacy classes of $G$ with respect to signed partitions $(n)$ and $({\bar n})$. 
\begin{enumerate}[label=\normalfont(\roman*)]
\item
If $n$ is even then $r\cdot G=H(X_1 \cup X_2)$ where $r=2(n-1)!$
\item
If $n$ is odd then $r\cdot G=HX_1$ and $r\cdot G=HX_2$ where $r=(n-1)!$.
\end{enumerate}
\end{theorem}
\begin{proof}
By Lemma \ref{l:4.1} it is enough to prove for $H$ as set $\{[1,1,\dots,1;\sigma],[-1,-1,\dots,-1;\rho] |\;\sigma,\rho \in S_n\}$.
We have \[ X_1= \left\{ [ b_1,b_2,\dots,b_n;\tau], \tau \in (1,2,\dots,n)^{S_n} \text{ and } \prod b_i=1\right\}, \]
\[ X_2=\left\{[ b_1,b_2,\dots,b_n;\tau], \tau \in (1,2,\dots,n)^{S_n} \text{ and } \prod b_i=-1\right\}.\]
We have $\mid X_1 \mid=\mid X_2 \mid=2^{n-1}(n-1)!$. Let $h_1=[1,1,\dots,1;\sigma] \in H$ and $x=[ b_1,b_2,\allowbreak \dots,b_n;\tau] \in X_1$. We have $h_1x=[ b_{\sigma(1)},b_{\sigma(2)},\dots,b_{\sigma(n)};\sigma\tau]$, and for $h_2=[-1,-1,\dots,-1;\sigma] \in H$, we get $h_2x=[ -b_{\sigma(1)},-b_{\sigma(2)},\allowbreak\dots,-b_{\sigma(n)};\sigma\tau]$. 

\text{Case {\rm (i).}} We have $\prod b_{\sigma(i)}=1$, and since $n$ is even, so $\prod (-1)^n b_{\sigma(i)}=1$, therefore $HX_1$ only contains those elements $[ a_1,a_2,\dots\allowbreak,a_n;\alpha]\in C_2\wr S_n$, with $\prod a_i=1$ and thus covers half the elements of $G$. Now we show that we can construct each such element of $G$, $2(n-1)!$ times. Let $g=[ a_1,a_2,\dots,a_n;\alpha] \in G$ with $\prod a_i=1$. Let $[b_1,a_2,\dots, b_n;\tau)] \in X_1$. Then there are $(n-1)!$ tuples $(\sigma,\tau)\in  S_n \times (1,2,\dots,n)^{S_n}$, satisfying $\alpha=\sigma\tau$. Moreover, for each given $g$ and for each tuple $(\sigma,\tau)$ we can construct $g$ twice by doing $[1,1,\dots,1;\sigma]\cdot [ b_1,b_2,\dots,b_n;\tau]=g$ and $[-1,-1,\dots,-1;\sigma]\cdot [b_1,b_2,\dots,b_n;\tau]=g$, that is, by taking $b_{\sigma(i)}=a_i$  and $b_{\sigma(i)}=-a_i$ respectively. Thus, we get
\[ HX_1=2(n-1)!\left\{[a_1,a_2,\dots,a_n;\sigma\;]\;\mid  \sigma \in S_n \text{ and } \prod a_i=1\right\}\]
Similarly, $HX_2$ only contains elements $[ a_1,a_2,\dots\allowbreak,a_n;\alpha]\in C_2\wr S_n$, with $\prod a_i=-1$ and covers the other half elements of $G$, $2(n-1)!$ times. So we have
\[ HX_2=2(n-1)!\left\{[\;a_1,a_2,\dots,a_n;\sigma\;]\;\mid  \sigma \in S_n \text{ and } \prod a_i=-1 \right\}\]
Thus, $H(X_1 \cup X_2)=2(n-1)!\cdot G$.

\text{Case {\rm (ii).}} Following the above case, we have $\prod b_{\sigma(i)}=1$, but since $n$ is odd, so $\prod (-1)^n b_{\sigma(i)}=-1$. Therefore, $HX_1$ covers all elements of the group $G$. For $[a_1,a_2,\dots,a_n,\alpha] \in G$, depending upon whether $\prod a_i$ is $1$ or $-1$ only one solution of either $b_{\sigma(i)}=a_i$ or $b_{\sigma(i)}=-a_i$ will be such that $\prod b_i=1$. Thus $HX_1$ covers all elements of $G$, $(n-1!)$ times. The case for $X_2$ follows similarly. Hence, the proof is complete.
\end{proof}

\begin{corollary}
Let $G$ be the finite Coxeter group of type $D_n$. 
\begin{enumerate}[label=\normalfont(\roman*)]
\item
If $n$ is even, then $r\cdot G= H(X_1\cup X_2)$ where $H \cong C_2 \times S_n$ is conjugate to $H_0$, $X_1,X_2$ are conjugacy classes of $G$ of signed partition $(n)$ and $r=2(n-1)!$.
\item
If $n$ is odd, then $r\cdot G= HX_1$ where $H \cong S_n$  is conjugate to $H_1$, $X_1$ conjugacy class $G$ of signed partition $(n)$ and $r=(n-1)!$.
\end{enumerate}

\end{corollary}
\begin{proof}
\text{Case {\rm (i).}} In this case, the conjugacy class of the signed partition $(n)$ splits into two conjugacy classes of $D_n$. Moreover, $H$ remains a subgroup of $D_n$. Thus from Theorem \ref{coxeter}, we immediately have $H(X_1 \cup X_2)=2(n-1)!\cdot D_n$. 

\text{Case {\rm (ii).}} For odd $n$, we take $H=\{[1,1,\dots,1;\sigma] \mid \sigma \in S_n\} \cong S_n$. Let $[a_1,a_2,\dots,a_n;\alpha] \in G$ and so $\prod a_i=1$. Following case (ii) of Theorem \ref{coxeter}, for each $(\sigma,\tau) \in S_n \times X_1$, we can uniquely choose $[b_1,b_2,\dots,b_n;\sigma] \in X_1$ such that $[1,1,\dots,1;\sigma]\cdot [b_1,b_2,\dots,b_n;\tau]=[a_1,a_2,\dots,a_n;\alpha]$. Thus $HX_1=(n-1)!\cdot G$.
\end{proof}

\section{Codes in Generalized Symmetric Groups}\label{Sec:5}

The generalized symmetric group is the wreath product $C_m \wr S_n= C_m^n \rtimes S_n$. We take $C_m=\{1,w,w^2,\dots,w^{m-1}\}$ and the action of $S_n$ on $C_m^n$ is defined as follows:
\[ \sigma\cdot (a_1,a_2,\dots,a_n)=(a_{\sigma(1)},a_{\sigma(2)},\dots, a_{\sigma(n)}) \]
The group multiplication is defined as
\[[a_1,a_2,\dots,a_n;\sigma]\cdot [b_1,b_2,\dots,b_n;\tau] =[a_1b_{\sigma(1)},a_2b_{\sigma(2)},\dots,a_nb_{\sigma(n)};\sigma\tau] \]

%\lambda_2^{s_{21}} w\lambda_2^{s_{2w}}\dots w^{m-1}\lambda_2^{s_{2w^{m-1}}}

Let $[a_1,a_2,\dots,a_n;\sigma] \in C_m \wr S_n$. Let $\sigma$ be written as product of disjoint cycles and $(k,\sigma(k),\dots,\sigma^{r}(k))$ be a cycle in $\sigma$. The product $a_ka_{\sigma(k)}\dots a_{\sigma^{r}(k)}$  is called the cycle product corresponding to that cycle of $\sigma$. The type of an element $[a_1,a_2,\dots,a_n;\sigma]$ is defined as the $m\times n$ matrix $(a_{ij})$ such that the entry $a_{ij}$ is the number of $j$-cycles in $\sigma$ with cycle product $w^{i-1}$. Two elements $[a_1,a_2,\dots,a_n;\sigma],[b_1,b_2,\dots,b_n;\tau] \in C_m \wr S_n$ are conjugate if they have the same cycle type. We refer to \cite[Section 3.7]{Ker71} for more details.

We use the following equivalent notation. Let $(\lambda_1,\lambda_2,\dots,\lambda_k)\vdash n$ and $(a_1,a_2,\dots,a_n) \in C_m^n$. We use symbol ${}_{w^j}\lambda_i$ to correspond to the $\lambda_i$-cycle $(k_1,k_2,\dots,k_{\lambda_i})$ such that for the tuple $(a_{k_1},a_{k_2},\dots,a_{k_{\lambda_i}})$ we have $\prod a_{k_i}=w^j$. For $j=0$, we use $\lambda$ and $ {}_1\lambda$ interchangeably. Let $s_{iw^j} \in \mathbb{N}$ for all $1\leq i \leq k$ and $0\leq j \leq m-1$. We call $ ( \lambda_1^{s_{11}} {}_w\lambda_1^{s_{1w}}\dots {}_{w^{m-1}}\lambda_1^{s_{1w^{m-1}}}\allowbreak \dots \lambda_k^{s_{k1}} {}_w\lambda_k^{s_{kw}}\dots {}_{w^{m-1}}\lambda_k^{s_{kw^{m-1}}} )$ an $m$-color partition of $n$. Two elements of $C_m \wr S_n$ are conjugate if they have the same $m$-color partitions and thus the conjugacy classes of $C_m \wr S_n$ are in one-to-one correspondence with the $m$-color partitions of $n$. 
%Let $[a_1,a_2,\dots,a_n;\sigma] \in S_n$.

For example $[1,w,1;(1,2)], [w^2,1,w;(1,2)] \in C_3\wr S_3$ correspond to $3$-color partitions $( 1^1 {}_w2^1)$ and $(_{w}1^1 {}_{w^2}2^1)$ respectively. We define subgroup $K_0$ of $G$ as set $\{[w^j,w^j,\dots,w^j;\sigma]|\; \sigma \in S_n,j=0,1,\dots,m-1\}$. We observe that $K_0 \cong C_m \times S_n$.

\begin{theorem} \label{generalized}
Let  $G=C_m \wr S_n$ and $H$ be a subgroup of $G$ conjugate to $K_0$. Let $X_{w^0},X_{w^1},\dots,X_{w^{m-1}}$ be the conjugacy classes of $G$ with respect to $m$-color partitions $(_{w^j}n)$ for $j=0,1,2,\dots,m-1$ respectively. 
Let ${\rm gcd}(m,n)=d$ and $m=dt$. Then $r\cdot G=H\cdot \bigcup\limits_{i \in I} X_i$ where $r=dl(n-1)!$, $l=1,2,\dots,t$ and $I$ is index set of size $dl$ constructed by taking $l$ many elements from each coset of $C_m/\langle w^d \rangle$.
\end{theorem}

\begin{proof}
By using Lemma \ref{l:4.1}, it is enough to prove for $H$ as set $\{[w^j,w^j,\dots,w^j;\sigma]|\;\sigma\in S_n,j=0,1,2,\dots,m-1\}$.
We have \[ X_{w^k}= \{ [ b_1,b_2,\dots,b_n;\tau ]| \tau \in (1,2,\dots,n)^{S_n} \text{ and } \prod b_i=w^k\}, \]
We have $\mid X_{w^k} \mid=m^{n-1}(n-1)!$. Let $h=[w^j,w^j,\dots,w^j;\sigma] \in H$. Let $x=[ b_1,b_2,\allowbreak \dots,b_n;\tau] \in X_{w^k}$ and so $\prod b_i=w^k$. We have $hx=[ w^jb_{\sigma(1)},w^jb_{\sigma(2)},\dots,w^jb_{\sigma(n)};\sigma\tau]$.

We have $\prod w^j b_{\sigma(i)}=w^{jn} \prod b_i=(w^{jn})w^k=w^{jn+k}$, therefore $HX_{w^k}$ only contains those elements $[ a_1,a_2,\dots\allowbreak,a_n;\alpha]\in C_m\wr S_n$, with $\prod a_i \in S= \{w^{k},w^{d+k},\dots,w^{(t-1)d+k}\}$ where $m=dt$. Now we show that we can construct each such element of $G$, $d(n-1)!$ times. Let fix $g=[ a_1,a_2,\dots,a_n;\alpha] \in G$ with $\prod a_i \in S$. Let $[b_1,b_2,\dots, b_n;\tau)] \in X_{w^k}$. Then there are $(n-1)!$ tuples $(\sigma,\tau)\in  S_n \times (1,2,\dots,n)^{S_n}$, satisfying $\alpha=\sigma\tau$. Now for each such $g$ and for each tuple $(\sigma,\tau)$ the equations $[w^j,w^j,\dots,w^j;\sigma]\cdot [b_1,b_2,\dots,b_n;\tau]=g$ has $d$ many solutions pairs $(j,x)$, by taking $w^jb_{\sigma(i)}=a_i$. Since we have $\prod a_i =w^{k+sd}$ with $s=0,1,2,\dots,t-1$, so for each $s$, the equation $jn \equiv sd \pmod m$ has $d$ solutions. Thus, we get
\[ HX_{w^k}=d(n-1)!\left\{[a_1,a_2,\dots,a_n;\sigma\;]\;\mid  \sigma \in S_n \text{ and } \prod a_i\in \{w^k, w^{d+k},\dots , w^{(t-1)d+k}\}  \right\} \]

We note that $HX_{w^j}=HX_{w^k}$ if and only if $w^j,w^k$ are in same coset of $C_m/\langle w^d \rangle$. Thus $H \cdot\bigcup\limits_{i\in I} X_i =dl(n-1)! \cdot G$ where $l=1,2,\dots,t-1$ and $I$ is index set of size $dl$ constructed by taking $l$ many elements from all cosets of $C_m/\langle w^d \rangle$.
\end{proof}

\textbf{Data availability statement:} This manuscript has no associated data.

\textbf{Acknowledgment.}
The authors thank Prof. Amit Kulshrestha for his interest and valuable suggestions in this work. The first named author acknowledges IISER Mohali for the PhD fellowship during this work. The second named author acknowledges IISER Mohali for the institute post-doctoral fellowship during this work.

\bibliography{ref}  

@article {GL20,
    AUTHOR = {Green, Holly M. and Liebeck, Martin W.},
     TITLE = {Some codes in symmetric and linear groups},
   JOURNAL = {Discrete Math.},
  FJOURNAL = {Discrete Mathematics},
    VOLUME = {343},
      YEAR = {2020},
    NUMBER = {8},
     PAGES = {111719, 5},
      ISSN = {0012-365X,1872-681X},
   MRCLASS = {05C25 (05C69 94B60)},
  MRNUMBER = {4108300},
MRREVIEWER = {Theodore\ C.\ Enns},
       DOI = {10.1016/j.disc.2019.111719},
       URL = {https://doi.org/10.1016/j.disc.2019.111719},
}

@article {FL26,
    AUTHOR = {Fang, Teng and Li, Jinbao},
     TITLE = {Proof of a conjecture of {G}reen and {L}iebeck on codes in
              symmetric groups},
   JOURNAL = {Discrete Math.},
  FJOURNAL = {Discrete Mathematics},
    VOLUME = {349},
      YEAR = {2026},
    NUMBER = {5},
     PAGES = {Paper No. 114999, 6},
      ISSN = {0012-365X,1872-681X},
   MRCLASS = {05C25 (05C69 94B60)},
  MRNUMBER = {5017563},
       DOI = {10.1016/j.disc.2026.114999},
       URL = {https://doi.org/10.1016/j.disc.2026.114999},
}

@book {macdonald1998symmetric,
    AUTHOR = {Macdonald, I. G.},
     TITLE = {Symmetric functions and {H}all polynomials},
    SERIES = {Oxford Classic Texts in the Physical Sciences},
   EDITION = {Second},
      NOTE = {With contribution by A. V. Zelevinsky and a foreword by
              Richard Stanley,
              Reprint of the 2008 paperback edition [MR1354144]},
 PUBLISHER = {The Clarendon Press, Oxford University Press, New York},
      YEAR = {2015},
     PAGES = {xii+475},
      ISBN = {978-0-19-873912-8},
   MRCLASS = {05E05 (01A75 05-02 20C30 20C33 20K01 33C80 33D80)},
  MRNUMBER = {3443860},
}

@misc{baker2016character,
  title={Character Polynomials and Row Sums of the Symmetric Group},
  author={Baker, Kirby and Early, Edward},
  note={\texttt{http://edwarde.create.stedwards.edu/rowsums.pdf}},
  year={2016}
}

@book{steinberg2012representation,
  title={Representation theory of finite groups: an introductory approach},
  author={Steinberg, Benjamin},
  volume={68},
  year={2012},
  publisher={Springer}
}

@article {Kra86,
    AUTHOR = {Kratochv\'{\i}l, Jan},
     TITLE = {Perfect codes over graphs},
   JOURNAL = {J. Combin. Theory Ser. B},
  FJOURNAL = {Journal of Combinatorial Theory. Series B},
    VOLUME = {40},
      YEAR = {1986},
    NUMBER = {2},
     PAGES = {224--228},
      ISSN = {0095-8956,1096-0902},
   MRCLASS = {94B60 (05C99)},
  MRNUMBER = {838221},
       DOI = {10.1016/0095-8956(86)90079-1},
       URL = {https://doi.org/10.1016/0095-8956(86)90079-1},
}

@book {Musili11,
    AUTHOR = {Musili, C.},
     TITLE = {Representations of finite groups},
    SERIES = {Texts and Readings in Mathematics},
    VOLUME = {3},
      NOTE = {Reprint of the 1993 original},
 PUBLISHER = {Hindustan Book Agency, New Delhi},
      YEAR = {2011},
     PAGES = {xviii+237},
      ISBN = {978-93-80250-18-2},
   MRCLASS = {20C15 (20-01 20C30)},
  MRNUMBER = {2857707},
}

@article {Feng_21,
    AUTHOR = {Feng, Tao and Li, Weicong and Zhou, Jingkun},
     TITLE = {On codes in the projective linear group {${\rm PGL}(2,q)$}},
   JOURNAL = {Finite Fields Appl.},
  FJOURNAL = {Finite Fields and their Applications},
    VOLUME = {75},
      YEAR = {2021},
     PAGES = {Paper No. 101812, 11},
      ISSN = {1071-5797,1090-2465},
   MRCLASS = {94B60 (05C25 11T06 20G40)},
  MRNUMBER = {4269558},
MRREVIEWER = {Felix\ Ulmer},
       DOI = {10.1016/j.ffa.2021.101812},
       URL = {https://doi.org/10.1016/j.ffa.2021.101812},
}

@book {Ker71,
    AUTHOR = {Kerber, Adalbert},
     TITLE = {Representations of permutation groups. {I}},
    SERIES = {Lecture Notes in Mathematics, Vol. 240},
 PUBLISHER = {Springer-Verlag, Berlin-New York},
      YEAR = {1971},
     PAGES = {v+192},
   MRCLASS = {20C30},
  MRNUMBER = {325752},
MRREVIEWER = {J.\ S.\ Frame},
}

@manual{GAP4,
    organization = "The GAP~Group",
    title        = "{GAP -- Groups, Algorithms, and Programming,
                    Version 4.16.0}",
    year         = 2026,
    url          = "\url{https://www.gap-system.org}",
    }
\bibliographystyle{plain} % Set style to amsalpha

\end{document}